\documentclass{CSML}
\pdfoutput=1

\usepackage{lastpage}
\newcommand{\dOi}{14(2:10)2018}
\lmcsheading{}{1--\pageref{LastPage}}{}{}%
{Jul.~14,~2016}{May ~16, 2018}{}

\usepackage{color}
\usepackage{hyperref}
\hypersetup{hidelinks}
\usepackage{amssymb}
\usepackage{graphicx,epsfig,epstopdf}

\theoremstyle{plain}
\theoremstyle{plain}
\theoremstyle{plain}

\newcommand{\lra}{\longrightarrow}

\newcommand{\ur}{\uparrow}
\newcommand{\dr}{\downarrow}

\theoremstyle{definition}
\newtheorem{dfn}[thm]{Definition}
\newtheorem{rmk}[thm]{Remark}
\newtheorem{exam}[thm]{Example}

\begin{document}

\title[Uniqueness of directed
complete posets]{Uniqueness of directed
complete posets based on Scott closed
set lattices}
\author[D. Zhao]{Zhao Dongsheng}	
\address{National Institute of Education, Nanyang Technological University, Singapore 637616}	
\email{dongsheng.zhao@nie.edu.sg}  
\author[L. Xu]{Xu Luoshan}	
\address{Department of Mathematics, Yangzhou University, Yangzhou 225002, China}	
\email{luoshanxu@hotmail.com}  





\begin{abstract}
 In analogy to a result due to Drake and Thron about topological spaces, this paper studies
the dcpos (directed complete posets) which  are fully determined, among all dcpos, by their lattices of all Scott-closed
subsets (such  dcpos  will be called  $C_{\sigma}$-unique).
  We   introduce the notions  of down-linear element and  quasicontinuous element in dcpos, and use them to prove that dcpos of certain classes, including all quasicontinuous dcpos as well as Johnstone's and Kou's examples, are $C_{\sigma}$-unique. As a consequence, $C_{\sigma}$-unique dcpos with their Scott topologies need  not be bounded sober.
\end{abstract}

\maketitle

\section{Introduction}

 From a result by Drake and Thron in \cite{drake-thron-1965}, one deduces the following  result (see Fact 3 in the Appendix): a topological space $X$ has the property that $C(X)$ isomorphic to $C(Y)$ implies $X$ is homeomorphic to $Y$ iff $X$ is sober and $T_D$ (every derived set $d(\{x\})=cl(\{x\})-\{x\}$ of point $x\in X$ is closed), where $C(X)$ and $C(Y)$ denote the lattices of closed sets of  $X$ and $T_0$ space $Y$, respectively (see also \cite{thornton-1973}, line 11-13, page 504).

For any dcpo $P$, let $C_{\sigma}(P)$ denote the lattice of all Scott closed subsets of $P$ (with the inclusion order).
A directed complete poset (or dcpo, for short)  $P$ will be called a $C_{\sigma}$-unique dcpo (or $C_{\sigma}$-unique, for short) if for any dcpo $Q$, $P$ is isomorphic to $Q$ whenever the lattices  $C_{\sigma}(P)$ and $C_{\sigma}(Q)$ are isomorphic. From a counterexample constructed in \cite{HJX-16} recently, we know that not every dcpo is $C_{\sigma}$-unique. It is therefore natural to ask which dcpos are $C_{\sigma}$-unique. One of the classic results in domain theory is that a dcpo $P$ is continuous iff the lattice $C_{\sigma}(P)$  is a completely distributive lattice (Theorem II-1.14 of \cite{comp}). From this it follows  that every continuous dcpo is sober and  $C_{\sigma}$-unique. In a similar way, one can deduce that every quasicontinuous dcpo is sober and $C_{\sigma}$-unique.
Compared  with  Drake's and Thron's result, one naturally asks whether every $C_{\sigma}$-unique dcpo is sober in their Scott topology.

In \cite{johnstone-81}, Johnstone constructed the first dcpo whose Scott topology is not sober. Later Isbell \cite{isbell-1982a} constructed a complete lattice whose Scott topology is  not sober. Kou \cite{kouhui} further gave a dcpo whose Scott topology is well-filtered but not sober. In this paper, we will introduce the concepts of quasicontinuous element and down-linear element in dcpos. With these concepts we  identify some classes of $C_{\sigma}$-unique dcpos, that  include all quasicontinuous dcpos as well as Johnstone's and Kou's examples. The full characterization of  $C_{\sigma}$-unique dcpos is still open.
\section{Preliminaries}
For any subset $A$ of a poset $P$, let $\ur\!\!A=\{x\in\!P: y\le\!x \mbox{ for some }y\in\!\!A\}$ and $\downarrow\!\!A =\{x\in\!P: x\le\!y \mbox{ for some } y\in\!A\}$. A subset $A$ is called an upper set if $A=\ur\!\!A$, and a lower set if $A=\downarrow\!\!A$.
A subset $U$ of a poset $P$ is Scott open  if (i) $U=\ur\!U$  and (ii) for
any directed subset $D$, $\bigvee D\in U$ implies $D\cap U\not=\emptyset$, whenever $\bigvee D$ exists.
All Scott open sets of a poset $P$ form a topology on $P$, denoted by $\sigma(P)$ and called the Scott topology on $P$. The complements of Scott open sets are called Scott closed sets. Clearly, a subset $A$ is Scott closed iff (i) $A=\dr\!\! A$ and (ii) for any directed subset $D\subseteq A$, $\bigvee D\in A$ whenever $\bigvee\!D$ exists. The set of all Scott closed sets of $P$ will be denoted by $C_{\sigma}(P)$. The space $(P, \sigma(P))$ is denoted by $\Sigma\!P$.

A poset $P$ is directed complete if its every directed subset has a supremum. A directed complete poset is briefly called a dcpo.

A subset $A$ of a topological space is irreducible if whenever $A\subseteq F_1\cup F_2$ with $F_1$ and $F_2$ closed, then  $A\subseteq F_1$ or $A\subseteq F_2$ holds. The set of all nonempty irreducible closed subsets of space $X$ will be denoted by $Irr(X)$.

For any $T_0$ topological space $(X, \tau)$, the specialization order $\le_{\tau}$ on $X$ is defined by $x\le_{\tau} y$ iff $x\in cl(\{y\})$ where $``cl(\cdot)"$ means taking closure.
\begin{rmk}

\label{irreducible sets}\leavevmode
\begin{enumerate}[label=(\arabic*)]
\item For any topological space $X$, $(Irr(X), \subseteq )$ is a dcpo. If  $\mathcal{D}$ is a directed subset  of $Irr(X)$, the supremum of $\mathcal{D}$ in $(Irr(X), \subseteq )$ equals $cl(\bigcup\mathcal{D})$ (the closure of $\bigcup\mathcal{D}$), which is the same as the supremum of $\mathcal{D}$ in the complete lattice  of all closed sets of $X$.

\item For any $x\in X$, $cl(\{x\})\in Irr(X)$. A $T_0$ space $X$ is called sober if $Irr(X)=\{cl(\{x\}): x\in X\}$, that is, every nonempty irreducible closed set is the closure of a point.

\item Assume that $(X,\tau)$ and $(Y,\eta)$ are topological spaces such that  the open set lattices ${(\tau,\subseteq)}$ and $(\eta,\subseteq)$ of $X$ and $Y$ are isomorphic, then the closed set lattice $(C(X), \subseteq )$ of $X$ and the closed set lattice $(C(Y), \subseteq)$ of $Y$ are also isomorphic (they are dual to the corresponding open set lattices). Since  irreducibility is a lattice-intrinsic property of elements, it follows that the posets $(Irr(X), \subseteq)$ and $(Irr(Y), \subseteq)$ are isomorphic.
\end{enumerate}
\end{rmk}

For a $T_0$ space $X$, a sobrification of $X$ is a sober space $Y$ together with a continuous mapping $\eta_X: X\lra Y$, such that for any continuous mapping $f: X\lra Z$ with $Z$ sober, there is a unique continuous mapping $\hat{f}: Y\lra Z$ satisfying $f=\hat{f}\circ \eta_X$. The sobrification of a $T_0$ space is unique up to homeomorphism. Clearly if a space $X$ is sober, then $X$  is homeomorphic to any sobrification of $X$.

\begin{rmk}
\label{soberfication}
The following facts about sober spaces and sobrifications are well-known.
\begin{enumerate}[label=(\arabic*)]
\item The set $Irr(X)$ of all nonempty closed irreducible sets of a $T_0$ space $X$ equipped with the hull-kernel topology is a sobrification  of $X$, where the mapping $\eta_X: X\lra Irr(X)$ is defined by $\eta_X(x)=cl(\{x\})$ for all $x\in X$.
 The closed sets of the hull-kernel topology consists of all sets of the form $h(A)=\{F\in Irr(X): F\subseteq A\}$ ($A$ is a closed set of $X$).
So the sobrification of a space $X$ is totally determined by the lattice $C(X)$.
(See Exercise V-4.9 of \cite{comp} for details, where the topology was given by means of open sets).

\item If $X$ and $Y$ are both sober spaces and the closed set lattices $C(X)$ and  $C(Y)$ are homomorphic,
then the sobrification  of $X$ and that of $Y$ are homeomorphic. Hence $X$ and $Y$ are homeomorphic.

\item From (1) and (2), we easily deduce that if $Y$ is a sober space, then $Y$ is a sobrification  of a $T_0$ space $X$  iff  the  closed  set lattices  $C(X)$ and $C(Y)$ are isomorphic (equivalently, the open set lattice of $Y$ is isomorphic to that of $X$).
\end{enumerate}
\end{rmk}

\smallskip

A  $T_0$ space $X$ will be called  Scott sobrifiable if there is a dcpo $P$ such that the Scott space $\Sigma \!P$ is a sobrification  of $X$.

For any $T_0$ space $(X, \tau)$, let $\le_{\tau}$ be the specialization order on $X$ ($x\le_{\tau} y$ iff $x\in cl(\{y\})$).
It is well-known that the specialization order on the Scott space $\Sigma\!P$ of a poset $P$ coincides with the original order on $P$. Thus a $T_0$ space $(X, \tau)$  is homeomorphic to $\Sigma\,P$ for some poset $P$ iff $(X, \tau)$ is homeomorphic to the Scott space $\Sigma(X, \le_{\tau})$ of the poset $(X, \le_{\tau})$.
The specialization order on the space $Irr(X)$ (with the hull-kernel topology) equals the inclusion order of sets.
From the above, we can easily deduce the following fact.

\begin{rmk}\label{scott soberfiable}

A $T_0$ space $(X, \tau)$ is Scott sobrifiable iff
for any Scott closed set  $\mathcal{F}$ of the dcpo $Irr(X)$,
there is a closed set $A$ of $X$ such that $\mathcal{F}=h(A)$, where $h(A)=\{F\in Irr(X): F\subseteq A\}$.
\end{rmk}

A topological space $(X, \tau)$ is called a d-space (or monotone convergence space) if (i) $X$ is $T_0$, (ii) the poset $(X, \le_{\tau})$ is a dcpo, and (iii) for any directed subset $D\subseteq (X, \le_{\tau})$, $D$ converges (as a net) to $\bigvee\!D$.

\begin{rmk}\label{d-spaces}\leavevmode
\begin{enumerate}[label=(\arabic*)]
\item Every sober space is a d-space.

\item Every Scott space $\Sigma P$ of a dcpo $P$ is a d-space.

\item If $(X, \tau)$ is a d-space, then every closed set $F$ of $X$ is a Scott closed set of the dcpo $(X, \leq_{\tau})$.
\end{enumerate}
\end{rmk}

\begin{lem}\label{sup of directed set of points}
Let $(X, \tau)$ be a d-space.
If $\{x_i: i\in I\}$ is a directed subset of $(X, \le_{\tau})$, then
the supremum $\sup\{cl(\{x_i\}): i\in I\}$ of $\{cl(\{x_i\}): i\in I\}$ in $Irr(X)$ equals
$cl(\{x\})$, where $x=\bigvee\{x_i: i\in I\}$.
\end{lem}
\noindent For more about dcpos, Scott topology and related topics we refer the reader to \cite{comp} and \cite{Jean-2013}.

\section{Main results}
In this section, we identify some classes of $C_{\sigma}$-unique dcpos, using irreducible sets, down-linear elements,
quasicontinuous elements and the M property, respectively.

A $T_0$ space is called  bounded-sober  if every nonempty upper bounded (with respect to the specialization order on $X$) closed irreducible subset of the space is the closure of a point \cite{zandf-2010}. Every sober space is bounded-sober, the converse implication is not true.
%
If $X$ is a $T_0$ space such that every  irreducible closed {\sl proper } subset is the closure of an element, then  $X$ is bounded-sober.
In the following, a dcpo whose Scott topology is sober (bounded-sober) will be simply called a sober (bounded-sober) dcpo.

\begin{lem}\label{non-sober bounded sober}
For a bounded-sober dcpo $P$, $\Sigma\!P$ is Scott sobrifiable if and only if $P$ is sober.
\end{lem}
\begin{proof} We only need to check  that  if $\Sigma P$ is not sober, then it is not  Scott sobrifiable.
Since $\Sigma\! P$ is not sober, there is a nonempty irreducible closed set $F$ such that $F$ is not the closure of any point.
From the assumption that $\Sigma P$ is bounded-sober, one can verify that the set $\mathcal{F}=\downarrow_{Irr(\Sigma P)}\{cl(\{x\}): x\in F\}$  consists precisely of the elements $cl(\{x\})$  ($x\in F$), and is a Scott closed set of $Irr(\Sigma P)$. But  any closed set $B$ of $\Sigma P$ containing all $cl(\{x\}) (x\in F)$ must contain $F$, thus $h(B)\not=\mathcal{F}$.  By Remark \ref{scott soberfiable},  $\Sigma P$ is not Scott sobrifiable.
\end{proof}

In the following, we shall write $P\cong Q$ if the two posets $P$ and $Q$ are isomorphic.
\begin{thm}\label{tm-sober-bounded sober}
Let  $P$ be a sober dcpo. For any bounded-sober dcpo $Q$, if $C_{\sigma}(P)\cong C_{\sigma}(Q)$ then  $P \cong Q$.
\end{thm}
\begin{proof}
Let $Q$ be a bounded-sober dcpo such that $C_{\sigma} (P)\cong C_{\sigma} (Q)$. Then, by Remark \ref{soberfication} (3),  $\Sigma\! P$ is a sobrification  of $\Sigma Q$. Thus $\Sigma Q$ is Scott sobrifiable. By Lemma \ref{non-sober bounded sober}, $\Sigma Q$  is also sober.  Therefore, by Remark \ref{soberfication} (2),   $\Sigma P$ and $\Sigma Q$ are homeomorphic, which then implies $P\cong Q$.
\end{proof}

\begin{dfn} An element $a$ of a poset $P$ is called down-linear if the subposet $\dr\!\! a=\{x\in P: x\le a\}$ is a chain (for any $x_1, x_2\in \dr\!\! a$, it holds that either $x_1\le x_2$ or $x_2\le x_1$).
\end{dfn}

\begin{lem}\label{irr set with liear down set}
Let $(X, \tau)$ be a d-space.
\begin{enumerate}[label=(\arabic*)]
\item If  $F\in Irr(X)$ is a down-linear element of the poset $Irr(X)$, then there exists an  $x\in X$ such that $F=cl(\{x\})$.

\item If $F\in Irr(X)$ equals the supremum of a directed set of  down-linear elements of $Irr(X)$, then $F=cl(\{x\})$ for some $x\in X$.
\end{enumerate}
\end{lem}
\begin{proof}\leavevmode
\begin{enumerate}[label=(\arabic*)]
\item First, the set $\{cl(\{x\}): x\in F\}$ is a subset of $\dr\!\! F$ in $Irr(X)$, so it is a chain. Thus $\{x: x\in F\}$ is a chain of $(X, \le_{\tau})$. Since $X$ is a d-space, $\widehat{x}=\sup\{x: x\in F\}$ exists.
Then, noticing that $F$ is closed, we have $\widehat{x}\in F$ by Remark \ref{d-spaces} (3). Then $F\subseteq cl(\{\widehat{x}\})\subseteq F$, implying $cl(\{\widehat{x}\})=F$.

\item Let $F$ be the supremum of a directed set of down-linear irreducible closed sets in $Irr(X)$.  Then by (1),  $F=\sup\{cl(\{x_i\}): i\in I\}$ in $Irr(X)$,
 where $\{cl(\{x_i\}): i\in I\}$ is a directed family. Thus, $\{x_i: i\in I\}$
 is a directed set of  $(X, \le_{\tau})$. Again, as $X$ is a d-space,
 $x=\sup\{x_i: i\in I\}$ exists. By Lemma \ref{sup of directed set of points},
 $cl(\{x\})=\sup\{cl(\{x_i\}): i\in I\}=F$. \qedhere
\end{enumerate}
\end{proof}
\noindent In the following, for a dcpo $P$, we shall use $Irr_{\sigma}(P)$ to denote the dcpo of all nonempty irreducible Scott closed subsets of $P$.
Without specification, irreducible sets of a poset mean the irreducible sets with respect to the Scott topology.

\begin{thm}\label{main theorem}
Let $P$ be a dcpo  satisfying the following condition
\begin{description}
\item[(DL-sup)] for any \it{proper} irreducible Scott closed set $F$, $F$  is either a down-linear element of $Irr_{\sigma}(P)$  or it is the supremum of a directed set of down-linear elements of $Irr_{\sigma}(P)$.
\end{description}
Then $P$ is $C_{\sigma}$-unique.
\end{thm}
\begin{proof}
Let dcpo $P$ satisfy the above condition (DL-sup) and $Q$ be a dcpo such that $C_{\sigma}(P)\cong~C_{\sigma}(Q)$.
\begin{enumerate}[label=(\arabic*)]
\item By Lemma \ref{irr set with liear down set}, if $F\in Irr_{\sigma}(P)$ and $F\not=P$, then $F=cl(\{x\})$ for some point.

\item Since $C_{\sigma}(P)\cong C_{\sigma}(Q)$, $Q$ also satisfies condition  (DL-sup). So every nonempty closed irreducible proper subset of $\Sigma\,Q$ is  the closure of a point.

\item Let $F$ be a nonempty irreducible closed subset of $P$ with an upper bound $a$. If $F\not=P$, then $F$ is the closure of some point by (1). Otherwise $F=P$, thus $a\in P$ is the largest element in $P$, hence $F=P=\dr\!a=cl(\{a\})$. Therefore $\Sigma P$ is bounded-sober. Similarly $\Sigma Q$ is bounded-sober.
If either $\Sigma P$ or $\Sigma Q$ is sober, then by Theorem
\ref{tm-sober-bounded sober}, $P\cong Q$. Assume now that neither  $\Sigma\!P$
nor $\Sigma\!Q$  is sober. Then there is a  nonempty irreducible closed set $F$
of $P$, which is not the closure of a singleton set. But by (1) and (2),  $F$
cannot be a proper subset, so $F=P$. Thus $P$ is an irreducible closed set which
does not equal to the closure of any singleton set. Similarly, $Q$ is an
irreducible closed set which is not the closure of any singleton set. Note that
in this case, $P$ and $Q$ are the top  elements of $Irr_{\sigma}(P)$ and
$Irr_{\sigma}(Q)$, respectively. Thus $Q\cong\{cl(\{y\}): y\in Q\}\cong
Irr_{\sigma}(Q)-\{Q\}\cong Irr_{\sigma}(P)-\{P\}\cong\{cl(\{x\}): x\in P\}\cong
P$, as desired. \qedhere
\end{enumerate}

\end{proof}

\begin{exam}\label{ex-JS-non-sober}
In \cite{johnstone-81}, Johnstone constructed the first non-sober dcpo  as $X=\mathbb{N}\times(\mathbb{N}\cup \{\infty\})$ with partial order defined by
$$ (m, n)\le (m', n') \Leftrightarrow \mbox{either } m=m' \mbox{ and } n\le n'\\
\mbox{or } n'=\infty \mbox{ and } n\le m'.$$
Then
\begin{enumerate}[label=(\alph*)]
\item $(X, \leq)$ is a dcpo, $X$ is irreducible and $X\not=cl(\{x\})$ for any $x\in X$.

\item If $F$ is a proper irreducible Scott closed set of $X$, then $F=\dr\!\!(m, n)$ for some $(m, n)\in X$.

\item If $n\not=\infty$, $\dr\!\!(m, n)$ is a down-linear element of $Irr_{\sigma}(X)$. If $n=\infty$, then $\dr (m, n)$ is the supremum of the
chain $\{\dr\!\!(m, k): k \not=\infty\}$ whose members are down-linear.
\end{enumerate}
Hence by Theorem \ref{main theorem}, we deduce that  dcpo $X=\mathbb{N}\times(\mathbb{N}\cup \{\infty\})$ is $C_{\sigma}$-unique.
Thus an $C_{\sigma}$-unique dcpo need not be sober.
\end{exam}

Next, we provide  a  class of $C_{\sigma}$-unique dcpos via quasicontinuous elements.

\begin{rmk}[cf.~\cite{lawxu}]\label{xu-lawson}
Let $A$ be a nonempty Scott closed set of a dcpo $P$. Then
\begin{enumerate}[label=(\roman*)]
\item $A$ is a dcpo.

\item For any subset $B\subseteq A$, $B$ is a Scott closed set of dcpo $A$ iff it is a Scott closed set of $P$. Thus $C_{\sigma}(A)=\dr_{C_{\sigma}(P)}A=\{B\in C_{\sigma}(P): B\subseteq A\}$.
\end{enumerate}
\end{rmk}

\noindent A finite subset $F$ of a dcpo $P$ is way-below an element $a\in P$, denoted by $F\ll a $, if for any directed subset $D\subseteq P$, $a\le \bigvee D$ implies $D\cap\!\!\ur\!\! F\not=\emptyset$. A dcpo $P$ is quasicontinuous if for any $x\in P$, the family
$$ \mathit{fin}(x)=\{F: F \mbox{ is finite and } F\ll x\}$$
is a directed family (for any $F_1, F_2\in \mathit{fin}(x)$ there is $F\in \mathit{fin}(x)$ such that $F\subseteq \ur F_1\cap {\ur F_2}$)  and
for any $x\not\leq y $ there is $F\in \mathit{fin}(x)$ satisfying $y\not\in \ur F$ (see Definition III-3.2 of \cite{comp}). Every continuous dcpo is quasicontinuous.

Every quasicontinuous dcpo is sober (Proposition III-3.7 of \cite{comp}). A dcpo $P$ is quasicontinuous iff the Scott open set lattice $\sigma(P)$  of $P$ is hypercontinuous (Theorem VII-3.9 of \cite{comp}). Assume that $P$ is a quasicontinuous dcpo and $Q$ is a dcpo such that $C_{\sigma}(P)$ is isomorphic to $C_{\sigma}(Q)$. Then
$\sigma(P)$ (it is dual to $C_{\sigma}(P)$) is isomorphic to $\sigma(Q)$(it is dual to $C_{\sigma}(Q)$, thus $\sigma(Q)$ is also hypercontinuous, implying that $Q$ is quasicontinuous. Thus both $\Sigma\!P$ and $\Sigma\!Q$ are sober spaces and they have isomorphic closed  set lattices, hence by Theorem \ref{tm-sober-bounded sober}, we have $P\cong Q$.
From this  we obtain  the following lemma.

\begin{lem}\label{quasicontinuous is scl-faithul}
Every quasicontinuous dcpo is $C_{\sigma}$-unique.
\end{lem}
\smallskip

An element $x$ of a dcpo $P$ is called a quasicontinuous element if the sub-dcpo $\dr\!\! x$ is a quasicontinuous dcpo.

\begin{thm}\label{quasicontinuous elements}
 Let $P$ be a dcpo. Then $P$ is $C_{\sigma}$-unique if it satisfies the following two conditions:
\begin{enumerate}[label=(\arabic*)]
\item $\Sigma P$ is bounded sober;

\item every element of $P$ is the supremum of a directed set of quasicontinuous elements.
\end{enumerate}
\end{thm}
\begin{proof}
Assume that $P$ is a dcpo satisfying the two conditions. Let $Q$ be a dcpo and $F: C_{\sigma}(P)\lra C_{\sigma}(Q)$ be an isomorphism. Then $F$ restricts to an isomorphism $F: Irr_{\sigma}(P)\lra Irr_{\sigma}(Q)$.
\begin{enumerate}[label=(\arabic*)]
\item Let $x\in P$ be a quasicontinuous element. Then $F(\dr\!x)$ is in $C_{\sigma}(Q)$ and, by Remark \ref{xu-lawson}, $C_{\sigma}(\dr\!x)=\{B\in C_{\sigma}(P): B\subseteq \dr\! x\}={\dr_{C_{\sigma}(P)}(\dr\!x)}$ is isomorphic via $F$ to $\dr_{C_{\sigma}(Q)}F(\dr\!x)=\{E\in C_{\sigma}(Q): E\subseteq F(\dr\!\!x)\}=C_{\sigma}(F(\dr\!x))$ (all Scott closed sets of $F(\dr\!x)$). Since the dcpo $\dr\!x$ is quasicontinuous, it is $C_{\sigma}$-unique. Hence the dcpo $\dr\!x$ is isomorphic to the dcpo $F(\dr x)$, implying that there is a largest element in $F(\dr\!x)$, denoted by $f(x)$. Hence $F(\dr\!x)=\dr\!f(x)$. It is easily observable that the mapping $f$ is well defined on the set of quasicontinuous elements of $P$, and for any two quasicontinuous elements $x_1, x_2\in P$, $f(x_1)\le f(x_2)$ iff $x_1\le x_2$.

\item If $x\in P$ is the supremum of a directed set $\{x_i: i\in I\}$ of quasicontinuous elements $x_i$, then
 \[\begin{array}{lll}
 F(\downarrow\!x)&=&F(\sup_{Irr_{\sigma}(P)}\{\downarrow\!x_i: i\in I\})\\
 &=& \sup_{Irr_{\sigma}(Q)}\{F(\downarrow\!x_i): i\in I\}\\
 &=&\sup_{Irr_{\sigma}(Q)}\{\downarrow\!f(x_i): i\in I\}\\
 &=&\downarrow\!y_x, \end{array}\]
 where $y_x=\sup_{Q}\{f(x_i): i\in I\}$ and $f(x_i)$ is the element in $Q$ defined for quasicontinuous elements $x_i$ in (1). Let $f(x)=y_x$ again.

Thus we have a monotone mapping $f: P\longrightarrow Q$. Following that $F$ is an isomorphism, we have that $f(x_1)\ge f(x_2)$ iff $x_1\ge x_2$.
It remains  to show that $f$ is surjective.

\item If $y\in\dr\!\!f(P)$, then $\downarrow\!\!y\subseteq \dr\!\!f(x)=F(\downarrow\!\!x)$ for some $x\in P$. Since $F$ restricts to an isomorphism between the dcpos $Irr_{\sigma}(P)$ and $Irr_{\sigma}(Q)$, there is $H\in Irr_{\sigma}(P)$ such that  $H\subseteq \downarrow\!x$ and  $F(H)=\downarrow\!y$.  But $P$ is bounded-sober, so $H=\downarrow\!x'$ for some $x'\in P$. It follows that $y=f(x')$, implying $y\in f(P)$. Therefore
 $f(P)$ is a lower  set of $Q$. Also clearly $f(P)$ is closed under sups of directed set, so it is a Scott closed subset of $Q$.

\item Since $F$ is an isomorphism between the lattices $C_{\sigma}(P)$ and $C_{\sigma}(Q)$, $P$ and $Q$ are the top elements in the respective lattices, we have that   $Q=F(P)=F(\sup_{C_{\sigma}(P)}\{\downarrow\!x:  x\in P\})=\sup_{C_{\sigma}(Q)}\{F(\downarrow\!x): x\in P\}=\sup_{C_{\sigma}(Q)}\{\downarrow\!f(x): x\in P\}$.

  For each $x\in P$, $\downarrow\!f(x)\subseteq f(P)$ because  $f(P)$ is a Scott
  closed set of $Q$, it holds then that $sup_{C_{\sigma}(Q)}\{\downarrow\!f(x):
  x\in P\}\subseteq f(P)$. Therefore $Q\subseteq f(P)$, which implies $Q=f(P)$.
  Hence  $f$ is also surjective. The proof is thus completed. \qedhere
\end{enumerate}
\end{proof}

\noindent If $x\in P$ is a down-linear element of a dcpo $P$, then $\dr\!x$ is a chain, so it is continuous (hence quasicontinuous).

\begin{cor}\label{tm-c1-c2-faithful}
If $P$ is a dcpo satisfying the following conditions, then $P$ is $C_{\sigma}$-unique:
\begin{enumerate}[label=(\arabic*)]
\item $P$ is bounded-sober.

\item every element $a\in P$ is the supremum of a directed set of down-linear elements.
\end{enumerate}
\end{cor}

\begin{exam}\label{ex-Kou-non-sob}
In order to answer the question whether every well-filtered dcpo is sober posed by Heckmann~\cite{heckmann},  Kou~\cite{kouhui} constructed another non-sober dcpo $P$ as follows:

Let  $X=\{x\in\mathbb{R}: 0 < x \le 1\}$,  $P_0=\{(k, a, b)\in \mathbb{R}: 0 < k <1,  0 < b \le a \leq 1\}$  and
$$P=X\cup P_0.$$
Define the partial order $\sqsubseteq$ on $P$ as follows:
\begin{enumerate}[label=(\roman*)]
\item  for $x_1, x_2\in X$, $x_1\sqsubseteq x_2$ iff $x_1=x_2$;

\item $(k_1, a_1, b_1) \sqsubseteq (k_2, a_2, b_2)$ iff $k_1\le k_2, a_1=a_2$ and $b_1=b_2$.

\item $(k, a, b) \sqsubseteq x $ iff  $a=x$ or  $kb\le x < b$.
\end{enumerate}
If $u=(h, a, b)\in P_0$, then $\dr\!\!u=\{(k, a, b): k\leq h\}$ is a chain. If $u=x\in P_0$, then $u = \bigvee\{(k, x, x): 0<k<1\}$, where each $(k, x, x)$ is a down-linear element and $\{(k, x, x): 0<k<1\}$ is a chain. Thus $P$ satisfies (2) of Corollary \ref{tm-c1-c2-faithful}.

Let $F$ be an irreducible nonempty Scott closed set of $P$ with an upper bound $v$. If $v=(h, a, b)\in P_0$, then $F\subseteq \dr\!\!(h, a, b)=\{(k, a, b): k\leq h\}$. Take $m=\bigvee\{k: (k, a, b)\in F\}$. Then $F=\dr\!\!(m, a, b)$, is the closure of  point $(m, a, b)$.

Now assume that $F$ does not have an upper bound in $P_0$, then $v=x$ for some $x\in P_0$. If $v\not\in F$, then due to the irreducibility of $F$, there exist $a, b$ such that $F\subseteq \{(k,a, b): 0<k<1\}$, which will imply that $F$ has an upper bound of the form $(m, a, b)$, contradicting the assumption. Therefore $v\in F$, implying that $F=\dr\!\!v$ (note that $F=\dr\!\!F$ is a lower set) is the  closure of point $v$.
It thus follows that $P$ satisfies (1) as well. By Corollary \ref{tm-c1-c2-faithful}, $P$ is $C_{\sigma}$-unique.
\end{exam}


Next, we give another class of $C_{\sigma}$-unique dcpos.
In \cite{Ho-Zhao}, Ho and Zhao introduced the following notions.
\begin{dfn}\label{dn-beneath-c-compact}
  Let $L$ be a poset and $x, y \in L$. The element  $x$ is {\em beneath} $y$,
denoted by $x\prec y$, if for every nonempty Scott-closed set $S\subseteq L$ with
$\bigvee S$ existing, $y\le \bigvee S$   implies  $x\in S$.
An element $x$ of $L$ is called {\em C-compact} if $x\prec x$.
Let $\kappa(L)$  denote the set of all the C-compact elements of $L$.
\end{dfn}

Let $P$ be a poset and $A\subseteq P$ finite.  The set $mub(A)$ of  minimal upper bounds of $A$ is said to be complete, if for any upper bound $x$ of $A$, there exists $y\in mub(A)$ such that $y\le x$.
A poset $P$ is said to satisfy the property $m$, if for all finite sets $A\subseteq P$, $mub(A)$ is complete.
A poset $P$ is said to satisfy the property $M$, if $P$ satisfies the property $m$ and for all finite set $A\subseteq P$, $mub(A)$ is finite.

\begin{rmk}
Let $L$ be a complete lattice and  $a\in L$ be a C-compact element. If  $x, y\in L$ such that $a\leq x\vee y$,
then $a\leq \bigvee (\downarrow\!x\cup\!\downarrow\!y)$ and $\downarrow\!x\cup\!\downarrow\!y$ is Scott closed, so $a\in \downarrow\!x\cup\!\downarrow\!y$, implying $a\le x $ or $a\le y$. Thus $a$ is $\vee$-irreducible.
\end{rmk}

\begin{cor}
For any  dcpo $P$, $\kappa(C_{\sigma}(P))\subseteq Irr_{\sigma}(P)$. That is, all C-compact Scott closed sets are  irreducible.
\end{cor}

\begin{lemC}[\cite{He-Xu}]\label{lm-M-c-compact-ideal} Let $P$ be a dcpo. Then
\begin{enumerate}[label=(\arabic*)]
\item For all $x\in P$, $\downarrow\!x\in \kappa(C_{\sigma}(P))$.

\item If $P$ satisfies the property $M$, then  $A\in\kappa(C_{\sigma}(P))$ iff $A= \downarrow\!x$
for some $x\in P$.
\end{enumerate}
\end{lemC}

\begin{thm}
If $P$ is a dcpo satisfying the property $M$ and  the condition (2) in Corollary \ref{tm-c1-c2-faithful}, then $P$ is $C_{\sigma}$-unique.
\end{thm}
\begin{proof}
Let $P$ be a dcpo satisfying the condition (2) in Corollary \ref{tm-c1-c2-faithful} and the property $M$. Assume that $Q$ is a dcpo and there is an order isomorphism $H: C_{\sigma}(P)\rightarrow C_{\sigma}(Q)$.
 Then the restrictions $H: \kappa(C_{\sigma}(P))\rightarrow\kappa(C_{\sigma}(Q))$ and $H: Irr_{\sigma}(P)\rightarrow Irr_{\sigma}(Q)$ are all order isomorphisms.

For all $q\in Q$, by Lemma \ref{lm-M-c-compact-ideal}(1),  $\downarrow\!q\in \kappa(C_{\sigma}(Q))$,  then $H^{-1}(\downarrow\!q)={\downarrow\!x_q}$ for a unique $x\in P$ by Lemma \ref{lm-M-c-compact-ideal}(2).  Now  define a map $h': Q\to P$ such that $h'(q)=x_q$ iff $H^{-1}({\downarrow\!q})={\downarrow\!x_q}$. The mapping $h'$ is  monomorphic and order preserving since $H^{-1}$ is.  Note that $\kappa(C_{\sigma}(Q))\cong \kappa(C_{\sigma}(P))\cong P$ is a dcpo.

Now let $x$ be any element of $P$.
\begin{enumerate}[label=(\roman*)]
\item If $x$ is down-linear, then $H(\downarrow\!\!x)$ is a linear subset in $Q$ (if $y_1, y_2\in H(\downarrow\!\! x)$, then $h'(y_1), h'(y_2)\in\downarrow\!\!x$), and is Scott closed. The supremum  $\sup_Q H(\downarrow\!x)$ exists and is in $H(\downarrow\!x)$. Thus $H(\downarrow\!x)={\downarrow\!q_x}$ for some $q_x\in Q$.

\item If $x$ is not down-linear, then  $x$ is the supremum of a directed set $C$ of down-linear elements. Since  $H$ preserves sups in $\kappa(C_{\sigma}(P))$ and $\kappa(C_{\sigma}(Q))$, we have that
\[\begin{array}{lll}
 H(\downarrow\!x)&=&H(\downarrow\!\sup C)\\
                 &=&H(\sup_{\kappa(C_{\sigma}(P))} \{\downarrow\!c: c\in C\})\\
                 &=&\sup_{\kappa(C_{\sigma}(Q))} \{H(\downarrow\!c): c\in C\}\\
                 &=&\downarrow\!\sup_Q \{q_c:  \downarrow\!q_c=H(\downarrow\!c), c\in C\}\\
                 &=&\downarrow\!q_x,\end{array}\]
for some $q_x\in Q$.
\end{enumerate}
By these facts, we defined a mapping $h: P\to Q$ such that $h(x)=q_x$ iff $F(\downarrow\!x)=\downarrow\!q_x$. It is then easy to see that $h$ is  monomorphic and order preserving since $H$ is. In addition, it is easy to verify  that $h'$  is the inverse of $h$, hence $h$ is an order isomorphism between $P$ and $Q$, as desired.
\end{proof}
 \noindent Note that  Kou's and Johnstone's examples of non-sober dcpos do not have the property $M$.

\section{Remarks and some possible further work}
We close the paper with some additional remarks and problems for further exploration.
\begin{rmk}\leavevmode
\begin{enumerate}[label=(\arabic*)]
\item If $P$ is a $C_{\sigma}$-unique dcpo and $P^*$ is the dcpo obtained by adding a top element to $P$, then one can show that $P^*$ is also $C_{\sigma}$-unique.
Let $X$ be the dcpo of Johnstone. Then $X^*$ is $C_{\sigma}$-unique, but $X^*$ is not bounded sober ($X$ is an irreducible Scott closed set of $X^*$ which is not the closure of any point of $X^*$). Thus a $C_{\sigma}$-unique dcpo  need not be bounded sober. So, bounded sobriety is not a necessary condition for a dcpo to be  $C_{\sigma}$-unique.

\item Recently, Ho, Goubault-Larrecq, Jung and Xi \cite{HJX-16}  constructed a pair of non-isomorphic dcpos having isomorphic Scott topologies,  showing the existence of non-$C_{\sigma}$-unique dcpos. Their  counterexample also reveals that sobriety is not a sufficient condition for a dcpo to be $C_{\sigma}$-unique.
\end{enumerate}
In view of the above remarks, to identify larger classes of  $C_{\sigma}$-unique dcpos and  formulate a  full characterization of $C_{\sigma}$-unique dcpos will be our future work.
\end{rmk}
\noindent{\bf Acknowledgement}
\begin{enumerate}[label=(\arabic*)]
\item This work was supported by  NIE AcRF Project (RI 3/16 ZDS),  NSF of China (11671008) and University Science Research Project of Jiangsu Province (15KJD110006).

\item Two anonymous  referees have provided us with many helpful comments and constructive suggestions for improvement. We are very grateful to them.
\end{enumerate}

\appendix
\section{}

In this part, for reader's convenience we present some details of the proof of a result, essentially due to  Drake and Thron,  on spaces which are uniquely
determined (among all $T_0$ spaces) by means of their closed set lattices. (This part is requested by one of the referees).
In \cite{drake-thron-1965}, Drake and Thron proved the following result.

\begin{cor}\label{cor:app}
Every representation family of a $\mathcal{C}$-lattice $(\Gamma, \ge)$ has exactly one element iff every irreducible element of $\Gamma$ is strongly irreducible.
\end{cor}
Here a lattice is called a $\mathcal{C}$-lattice if it is isomorphic to  the lattice $C(X)$ of all closed sets of a topological space $X$.
An element $a$ of a  lattice $L$ is called irreducible (strongly irreducible) iff $a$ can not be expressed as the supremum of a finite (arbitrary) number of elements of $L$, which are strictly less than $a$.

By the definition of representation families of $\mathcal{C}$-lattices (see page 58 of \cite{drake-thron-1965}) we deduce the following fact, which is equivalent to the above Corollary~\ref{cor:app}:

\begin{fact}\label{fact:1}
A $T_0$ topological space $X$ has the property that (for any $T_0$ space $Y$) $C(X)$ isomorphic to $C(Y)$ implies $X$ is homeomorphic to $Y$
if and only if every irreducible closed set in the space  is strongly irreducible.
\end{fact}

A space $X$ is called a $T_D$ space iff for any $x\in X$, the derived set $d(\{x\})=cl(\{x\})-\{x\}$ is a closed set (see Definition 2.1 of \cite{thron-1962}). For example, every $T_1$ space is a $T_D$ space.

\begin{fact}\label{fact:2}
  A topological space is both sober and $T_D$ iff every irreducible closed set in the space  is strongly irreducible.
\end{fact}

\begin{proof} First note that for any $\{A_i: i\in I\}\subseteq C(X)$, the supremum $\bigvee_{C(X)}\{A_i: i\in I\}$ of $A_i's$ in the lattice $C(X)$ equals
$cl(\bigcup\{A_i: i\in I\})$.

Assume that the space $X$ is both sober and $T_D$. Let $F$ be an irreducible element of $C(X)$. Then $F=cl(\{x\})$ for some $x\in X$ because $X$ is sober.
Let $F=\bigvee_{C(X)}\{A_i: i\in I\}$ holds, where $A_i\in C(X) (i\in I)$. Then $\bigcup\{A_i: i\in I\}\subseteq F=cl(\{x\})$. Thus $A_i\subseteq cl(\{x\})$ for each $i$.
If $cl(\{x\})\not= A_i$ for every $i$, then $A_i\subseteq cl(\{x\})-\{x\}$, therefore $\bigcup\{A_i: i\in I\}\subseteq cl(\{x\})-\{x\}$. Since $cl(\{x\})-\{x\}$ is closed, we have
$F=cl(\bigcup\{A_i: i\in I\})\subseteq cl(\{x\})-\{x\}$, which contradicts $F=cl(\{x\})$. Hence $F=cl(\{x\})=A_i$ for $i$, showing that $F$ is strongly irreducible.

Now assume that every irreducible element of $C(X)$ is strongly irreducible.
Let $F$ be a non empty  irreducible member of $C(X)$. Then $F=\bigvee_{C(X)}\{cl(\{x\}): x\in F\}$, so $F=cl(\{x\})$ for some $x\in F$ because $F$ is strongly irreducible. It follows that $X$ is sober. Now let $x\in X$ be any element. Assume that $cl(\{x\})-\{x\}$ is not closed.
Then $cl(\{x\})-\{x\}$ is a proper subset of  $cl(cl(\{x\})-\{x\})$. But trivially $cl(cl(\{x\})-\{x\})\subseteq cl(\{x\})$, thus $cl(cl(\{x\})-\{x\})= cl(\{x\})$. Thus $cl(\{x\})=\bigvee_{C(X)}\{cl(\{y\}): y\in cl(\{x\})-\{x\}\}$. Since $cl(\{x\})$ is irreducible, it is strongly irreducible by the assumption, we have $cl(\{x\})=cl(\{y\})$ for some $y\in cl(\{x\})-\{x\}$, which is not possible because $X$ is $T_0$.
Therefore $cl(\{x\})-\{x\}$ must be closed. Hence $X$ is $T_D$.
\end{proof}

From Fact~\ref{fact:1} and Fact~\ref{fact:2} we derive the following result, first explicitly stated in \cite{thornton-1973} (page 504 line 11-13) with no proof (where sober spaces are called pc spaces).
\begin{fact}
A space $X$  has the property that (for any $T_0$ space $Y$) $C(X)$  isomorphic to
$C(Y)$ implies $X$ is homeomorphic to $Y$ iff $X$ is both sober and $T_D$.
\end{fact}

\end{document}